\documentclass{amsart}
\usepackage{graphics}
\usepackage{amsfonts,amsmath}
\begin{document}

 \newtheorem{thm}{Theorem}[section]
 \newtheorem{cor}[thm]{Corollary}
 \newtheorem{lem}[thm]{Lemma}{\rm}
 \newtheorem{prop}[thm]{Proposition}

 \newtheorem{defn}[thm]{Definition}{\rm}
 \newtheorem{assumption}[thm]{Assumption}
 \newtheorem{rem}[thm]{Remark}
 \newtheorem{ex}{Example}\numberwithin{equation}{section}

\def\br{\bar{\rho}}
\def\la{\langle}
\def\C{\mathcal{C}}
\def\ra{\rangle}
\def\e{{\rm e}}
\def\x{\mathbf{x}}
\def\by{\mathbf{y}}
\def\bz{\mathbf{z}}
\def\F{\mathcal{F}}
\def\R{\mathbb{R}}
\def\T{\mathbf{T}}
\def\N{\mathbb{N}}
\def\K{\mathbf{K}}
\def\bK{\overline{\mathbf{K}}}
\def\Q{\mathbf{Q}}
\def\M{\mathbf{M}}
\def\O{\mathbf{O}}
\def\C{\mathbf{C}}
\def\P{\mathbf{P}}
\def\Z{\mathbf{Z}}
\def\H{\mathbf{H}}
\def\A{\mathbf{A}}
\def\V{\mathbf{V}}
\def\AA{\overline{\mathbf{A}}}
\def\B{\mathbf{B}}
\def\c{\mathbf{c}}
\def\L{\mathbf{L}}
\def\bS{\mathbf{S}}
\def\H{\mathbf{H}}
\def\I{\mathbf{I}}
\def\Y{\mathbf{Y}}
\def\X{\mathbf{X}}
\def\G{\mathbf{G}}
\def\f{\mathbf{f}}
\def\z{\mathbf{z}}
\def\y{\mathbf{y}}
\def\d{\hat{d}}
\def\bx{\mathbf{x}}
\def\y{\mathbf{y}}
\def\v{\mathbf{v}}
\def\g{\mathbf{g}}
\def\w{\mathbf{w}}
\def\b{\mathcal{B}}
\def\a{\mathbf{a}}
\def\b{\mathbf{b}}
\def\q{\mathbf{q}}
\def\u{\mathbf{u}}
\def\h{\mathbf{h}}
\def\s{\mathcal{S}}
\def\bs{\mathbf{s}}
\def\br{\mathbf{r}}
\def\cc{\mathcal{C}}
\def\co{{\rm co}\,}
\def\cop{{\rm COP}\,}
\def\tg{\tilde{f}}
\def\tx{\tilde{\x}}
\def\supmu{{\rm supp}\,\mu}
\def\supnu{{\rm supp}\,\nu}
\def\m{\mathcal{M}}
\def\s{\mathcal{S}}
\def\k{\mathcal{K}}
\def\la{\langle}
\def\ra{\rangle}
\def\psd{{\rm Psd}}

\title[nonnegativity]{Borel measures with a density on a compact semi-algebraic set}

\author{Jean B. Lasserre}
\address{LAAS-CNRS and Institute of Mathematics\\
University of Toulouse\\
LAAS, 7 avenue du Colonel Roche\\
31077 Toulouse C\'edex 4,France}
\email{lasserre@laas.fr}
\date{}

\begin{abstract}
Let $\K\subset\R^n$ be a compact basic semi-algebraic set.
We provide a necessary and sufficient condition 
(with no {\it \`a priori} bounding parameter)
for a real sequence $\y=(y_\alpha)$, $\alpha\in\N^n$, 
to have a finite representing Borel measure
absolutely continuous w.r.t. the Lebesgue measure on $\K$, and
with a density in $\cap_{p\geq1} L_p(\K)$.
With an additional condition involving a bounding parameter,
the condition is necessary and sufficient for existence of a density in $L_\infty(\K)$.
Moreover, nonexistence of such a density
can be detected by solving finitely many of a hierarchy of semidefinite programs.
In particular, if the semidefinite program at step $d$ of the hierarchy has no solution then
the sequence cannot have a representing measure on $\K$ with a density in $L_p(\K)$ for any $p\geq 2d$.
\end{abstract}



\maketitle

\section{Introduction}


The famous Markov moment problem (also called the $L$-problem of moments) is concerned with characterizing
real sequences $(s_n)$, $n\in\N$, which are moment of a Borel probability measure $\mu$ on $[0,1]$
with a bounded density with respect to (w.r.t.) the Lebesgue measure. It  
was posed by Markov and later solved by Hausdorff with
the following necessary and sufficient condition:
\begin{equation}
\label{markov}
1=s_0\quad\mbox{and}\quad 0\leq\,s_{nj}\,\leq\,c/(n+1),\qquad \forall\,n,j\in\N,
\end{equation}
for some $c>0$, where $s_{nj}:=(-1)^{n-j}{n\choose j}\Delta^{n-j}s_j$, and $\Delta$ is the forward operator
$s_n\mapsto \Delta s_n=s_{n+1}-s_n$.
Similarly, with $p>1$, if in (\ref{markov}) one replaces the condition ``$s_{nj}\leq c/(n+1)$ for all  $n,j\in\N$", with the condition
\begin{equation}
\label{lp}
\sup_n\:\left(\frac{1}{n+1}\sum_{j=0}^n((n+1)\,s_{nj})^p\right)^{1/p}\,<\,c,\end{equation}
then one obtains a characterization of real sequences having a representing Borel measure with a density
in $L_p([0,1])$ with $p$-norm bounded by $c$.
For an illuminating discussion with historical remarks the reader is referred to Diaconis and Freedman \cite{Diaconis}
where the authors also make a connection with De Finetti's theorem on exchangeable 0-1 valued random variables. 
In addition, Putinar \cite{putinar1,putinar2} has provided  a characterization
of extremal solutions of the two-dimensional $L$-problem of moments.

Observe that the above condition (\ref{markov}) is stated in terms of {\em linear}
inequalities on the $s_j$'s. An alternative {\it if and only if} characterization is via 
positive definiteness of some sequence $t_n(c)$ related to the sequence $(s_n)$, as described in Ahiezer and Krein \cite{ahiezer}
for full (as well as truncated) classical and trigonometric (univariate) moment problems.
For instance, a sequence $(s_n)$, $n\in\N$, has a representing measure
on the real line $(-\infty,+\infty)$ with density $f$ (w.r.t. the Lebesgue measure) bounded by $c$ if and only if a certain sequence
$t_n(c)$, $n\in\N$, where each $t_n(c)$ is a polynomial of degree $n$ in the variables $(\frac{1}{c},s_0,\ldots,s_n)$, is positive definite.
See e.g. \cite[Theorem 6, p. 72]{ahiezer} 
and \cite[Theorem 7, p. 77]{ahiezer} for the truncated and full $L$-moment problems, respectively. 
In other words, each (Hankel) moment matrix of size $n$ (filled up with finitely many of the $t_n(c)$'s) is positive semidefinite.

In modern language those conditions are Linear Matrix Inequalities (LMIs) on the moment matrix associated 
with the sequence $(t_n(c))$ (but not $(s_n)$). More recently, another type of LMI conditions 
still in the spirit of those provided in Ahiezer and Krein \cite{ahiezer} (but now directly on the sequence $(s_n)$)
was provided in Lasserre \cite[p. 67--69]{lasserre-book} for the full $L$-moment
problem on arbitrary compact basic semi-algebraic sets $\K\subset\R^n$. 
It involves the moment and localizing matrices associated with 
$(s_n)$ and the Lebesgue measure
$\lambda$ on $\K$, respectively. It states that the moment and localizing matrices
associated with the sequence $(s_n)$ and $\K$
must be dominated by that of $\lambda$ (scaled with some factor $c>0$). 
A similar characterization also holds for measures on non-compact sets, satisfying an additional Carleman type condition.
Hence such conditions are LMIs on the $s_j$'s directly with no need to build up a related sequence $(t_n(c))$ 
which is nonlinear in the $s_j$'s, as in \cite{ahiezer}.

However, all the above characterizations have a common feature. They all
involve the {\it \`a priori} unknown scalar $c>0$ which is
precisely the required bound on the density. So for instance, 
if the conditions (\ref{markov}) fail one does not know whether 
it is because the sequence has no representing measure $\mu$ 
with a density in $L_\infty([0,1])$ or because $c$ is not large enough.
So it is more appropriate to state that (\ref{markov}) are necessary and sufficient conditions for
$\mu$ to have a density $f\in L_\infty([0,1])$ {\it with} $\Vert f\Vert_\infty\leq c$. And similarly for
(\ref{lp}) for $f\in L_p([0,1])$ {\it with} $\Vert f\Vert_p\leq c$.\\

\noindent
{\bf Contribution:}  Consider a compact basic semi-algebraic set $\K\subset\R^n$
of the form
\begin{equation}
\label{setk}
\K:=\{\x\in\R^n\,:\, g_j(\x)\geq0,\: j=1,\ldots,m\},
\end{equation}
for some polynomials $g_j\in\R[\x]$, $j=1,\ldots,m$.
For every, $p\in\N$, denote by $L_p(\K)$ the Lebesgue space of functions such that $\int_\K \vert f\vert^p \lambda(d\x)<\infty$. We then provide a set of conditions (S) with no {\it \`a priori} bound $c$, and such that:\\

- A real sequence $\y=(y_\alpha)$, $\alpha\in\N^n$, has a finite representing Borel measure with a density in $\cap_{p\geq1} L_p(\K)$, if and only if (S) is satisfied. In particular, if (S) is violated
we obtain a condition (with no \`a priori bounding parameter $c$) for non existence of a density in
$\cap_{p\geq1} L_p(\K)$ (hence no density in $L_\infty(\K)$ either).

- A real sequence $\y=(y_\alpha)$, $\alpha\in\N^n$, has a finite representing Borel measure with a density in $L_\infty(\K)$ if and only if (S) and an additional condition (involving an {\it \`a priori} bound $c>0$), 
are satisfied.

In addition, the conditions (S) consist of a hierarchy of Linear Matrix Inequalities
(LMIs) (again in the spirit of \cite{ahiezer,lasserre-book}) and so can be tested numerically via available semidefinite programming softwares.
In particular, if a finite Borel measure $\mu$ does not have a density 
in $\cap_{p\geq1}L_p(\K)$ (and so no density in $L_\infty(\K)$ either), it can be detected by solving {\it finitely many} semidefinite programs in the hierarchy until one has no feasible solution. 
That is, it can be detected from finitely many of its moments. This is illustrated on a simple example.

Conversely, if the semidefinite program at step $d$ of the hierarchy has no solution, then one may
conclude that  the real sequence $\y$ cannot have a representing measure on $\K$ with a density in $L_{p}(\K)$,
for any $p\geq 2d$.

So a distinguishing feature of our result is the absence of an {\it \`a priori} bound $c$
in the set of condition (S) to test whether $\y$ has a density in $\cap_{p=1}^\infty L_p(\K)$. 
Crucial for our result is a representation of polynomials that are positive on 
$\K\times\R$, by Powers \cite{powers}; see also Marshall \cite{marshall1,marshall2}.

\section{Main result}

\subsection{Notation, definitions and preliminary results}

Let $\R[\x,t]$ (resp. $\R[\x,t]_d$) denote the ring of real polynomials in the variables
$\x=(x_1,\ldots,x_n,t)$ (resp. polynomials of degree at most $d$), whereas $\Sigma[\x,t]$ (resp. $\Sigma[\x,t]_d$) denotes 
its subset of sums of squares (s.o.s.) polynomials (resp. of s.o.s. of degree at most $2d$).
For every
$\alpha\in\N^n$ the notation $\x^\alpha$ stands for the monomial $x_1^{\alpha_1}\cdots x_n^{\alpha_n}$ and for every $d\in\N$, let $\N^{n+1}_d:=\{\beta\in\N^{n+1}:\sum_j\beta_j\leq d\}$ whose cardinal is $s(d)={n+d+1\choose d}$.
A polynomial $f\in\R[\x,t]$ is written 
\[(\x,t)\mapsto f(\x,t)\,=\,\sum_{(\alpha,k)\in\N^n\times\N}\,f_{\alpha k} \,\x^\alpha\,t^k,\]
and $f$ can be identified with its vector of coefficients $\f=(f_{\alpha k})$ in the canonical basis 
$(\x^\alpha,t^k)$, $(\alpha,k)\in\N^n\times\N$, of $\R[\x,t]$. But we can also write $f$ as
\begin{equation}
\label{decomp-f}
(\x,t)\mapsto f(\x,t)\,=\,\sum_{k\in\N}\,f_k(\x)\,t^k,\end{equation}
for finitely many polynomials $f_k\in\R[\x]$.\\

A real sequence $\z=(z_{\alpha k})$, $(\alpha,k)\in\N^n\times\N$, has a {\it representing measure} if
there exists some finite Borel measure $\nu$ on $\R^n\times\R$ such that 
\[z_{\alpha k}\,=\,\int_{\R^{n+1}}\x^\alpha\,t^k\,d\nu(\x,t),\qquad\forall\,(\alpha,k)\in\N^n\times\N.\]

Given a real sequence $\z=(z_{\alpha k})$ define the linear functional $L_\y:\R[\x,t]\to\R$ by:
\[f\:(=\sum_{\alpha,k} f_{\alpha k}\,\x^\alpha\,t^k)\quad\mapsto L_\z(f)\,=\,\sum_{\alpha,k}f_{\alpha k}\,y_{\alpha k},\qquad f\in\R[\x,t].\]
\subsection*{Moment matrix}
The {\it moment} matrix associated with a sequence
$\z=(z_{\alpha k})$, $(\alpha,k)\in\N^n\times\N$, is the real symmetric matrix $\M_d(\z)$ with rows and columns indexed by $\N^{n+1}_d$, and whose entry $(\alpha,\beta)$ is just $z_{\alpha+\beta}$, for every 
$\alpha,\beta\in\N^{n+1}_d$. 
Alternatively, let
$\v_d((\x,t))\in\R^{s(d)}$ be the vector $((\x,t)^\alpha)$, $\alpha\in\N^{n+1}_d$, and
define the matrices $(\B_\alpha)\subset\s^{s(d)}$ by
\begin{equation}
\label{balpha}
\v_d((\x,t))\,\v_d((\x,t))^T\,=\,\sum_{\alpha\in\N^{n+1}_{2d}}\B_\alpha\,(\x,t)^\alpha,\qquad\forall(\x,t)\in\R^{n+1}.\end{equation}
Then $\M_d(\z)=\sum_{\alpha\in\N^{n+1}_{2d}}z_\alpha\,\B_\alpha$.

If $\z$ has a representing measure $\nu$ then
$\M_d(\z)\succeq0$ because
\[\langle\f,\M_d(\z)\f\rangle\,=\,\int f^2\,d\nu\,\geq0,\qquad\forall \,\f\,\in\R^{s(d)}.\]

\subsection*{Localizing matrix}
With $\z$ as above and $g\in\R[\x,t]$ (with $g(\x,t)=\sum_\gamma g_\gamma(\x,t)^\gamma$), the {\it localizing} matrix associated with $\z$ 
and $g$ is the real symmetric matrix $\M_d(g\,\z)$ with rows and columns indexed by $\N^{n+1}_d$, and whose entry $(\alpha,\beta)$ is just $\sum_{\gamma}g_\gamma z_{\alpha+\beta+\gamma}$, for every $\alpha,\beta\in\N^{n+1}_d$.
Alternatively, let $\C_\alpha\in\s^{s(d)}$ be defined by:
\begin{equation}
\label{calpha}
g(\x,t)\,\v_d(\x,t)\,\v_d(\x,t)^T\,=\,\sum_{\alpha\in\N^{n+1}_{2d+{\rm deg}\,g}}\C_\alpha\,(\x,t)^\alpha,\qquad\forall(\x,t)\in\R^{n+1}.\end{equation}
Then $\M_d(g\,\z)=\sum_{\alpha\in\N^{n+1}_{2d+{\rm deg}g}}z_\alpha\,\C_\alpha$.

If $\z$ has a representing measure $\nu$ whose support is 
contained in the set $\{(\x,t)\,:\,g(\x,t)\geq0\}$ then
$\M_d(g\,\z)\succeq0$ because
\[\langle\f,\M_d(g\,\y)\f\rangle\,=\,\int f^2\,g\,d\nu\,\geq0,\qquad\forall \,\f\,\in\R^{s(d)}.\]

With $\K$ as in (\ref{setk}), and for every $j=0,1,\ldots,m$, let $v_j:=\lceil ({\rm deg}\,g_j)/2\rceil$.
\begin{defn}
With $\K$ as in (\ref{setk}) let $P(g)\subset\R[\x,t]$ 
be the convex cone:
\begin{equation}
\label{put1}
P(g)\,=\,\left\{\:\sum_{\beta\in\{0,1\}^m} \psi_\beta(\x,t)\,g_1(\x)^{\beta_1}\cdots g_m(\x)^{\beta_m}\::\: \psi_\beta\in\Sigma[\x,t]\,
\right\}.\end{equation}
\end{defn}
The convex cone $P(g)$ is called a {\it preordering} associated with the $g_j$'s.\\

\begin{prop}
\label{prop1}
Let $\K$ be as in (\ref{setk}). A polynomial $f\in\R[\x,t]$ is nonnegative on $\K\times\R$ only if $f$ can be written as
\begin{equation}
\label{f-decomp}
(\x,t)\,\mapsto\quad f(\x,t)\,=\,\sum_{k=0}^{2d}f_k(\x)\,t^{k},\end{equation}
for some $d\in\N$ and where $f_{2d}\geq0$ on $\K$.
\end{prop}
\begin{proof}
Suppose that the highest degree in $t$ is $2d+1$ for some $d\in\N$. Then $f_{2d+1}\neq0$ and so
by fixing an arbitrary $\x_0\in\K$, the univariate polynomial $t\mapsto f(\x_0,t)$ 
can be made negative, in contradiction with $f\geq0$ on $\K\times\R$. Hence the highest degree in $t$ is even, say $2d$.
But then of course, for obvious reasons $f_{2d}\geq0$ on $\K$. 
\end{proof}
We have the following important preliminary result.
\begin{thm}[\cite{marshall1,powers}]
\label{th-mar}
Let $\K$ as in (\ref{setk}) be compact and let
$f\in\R[\x,t]$ be of the form $f(\x,t)=\sum_{k=0}^{2d}f_k(\x)t^k$ for some polynomials
$(f_k)\subset\R[\x]$, and with $f_{2d}>0$ on $\K$. Then $f\in P(g)$ if $f>0$ on $\K\times\R$.
\end{thm}
And so we can derive a version of the $\K\times\R$-moment problem where for each $\beta\in\N^m$,
the notation $g^\beta$ stands for the polynomial $g_1^{\beta_1}\cdots g_m^{\beta_m}$.
\begin{cor}
\label{mar-moment}
Let $\K$ as in (\ref{setk}) be compact. A real sequence
$\z=(z_{\alpha k})$, $(\alpha,k)\in\N^n\times\N$, has a representing measure on $\K\times\R$ if and only if
\begin{equation}
\label{strip-moment}
\M_d(\z)\succeq0;\quad \M_d(g^\beta\,\z)\succeq0,\qquad \beta\in\{0,1\}^m,\end{equation}
for every $d\in\N$.
\end{cor}
\begin{proof}
The {\it only if} part is straightforward from the definition of the moment and localizing matrix
$\M_d(\z)$ and $\M_d(g^\beta\,\z)$, respectively.

The {\it if part}. Suppose that (\ref{strip-moment}) holds true, and let $f\in\R[\x,t]$
be nonnegative on the closed set $\K\times\R$. Hence by Proposition \ref{prop1}, $f$ has the decomposition (\ref{f-decomp}) for some 
integer $d\neq0$. For every $\epsilon>0$,
the polynomial $(\x,t)\mapsto f_\epsilon(\x,t):=f(\x,t)+\epsilon(1+t^{2d})$ has the decomposition
\[f_\epsilon(\x,t)\,=\,\sum_{k=0}^{2d}f_{\epsilon k}(\x)\,t^k,\]
with $f_{\epsilon 0}=f_0+\epsilon$ and $f_{\epsilon 2d}(\x)=f_{2d}(\x)+\epsilon $.
Therefore, $f_\epsilon$ is strictly positive on
$\K\times\R$, and $f_{\epsilon 2d}>0$ on $\K$. By Theorem \ref{th-mar}, $f_\epsilon\in Q(g)$, i.e.,
\[f_\epsilon(\x,t)\,=\,\sum_{\beta\in\{0,1\}^m}\psi_\beta(\x,t)\,g(\x)^\beta,\]
for some SOS polynomials $(\psi_\beta)\subset\Sigma[\x,t]$. Next, let $\z$ satisfy (\ref{strip-moment}).
Then
\begin{eqnarray*}
L_\z(f)+\epsilon \,L_\y(1+t^{2d})&=&L_\z(f_\epsilon)\\
&=&\sum_{\beta\in\{0,1\}^m}L_\z(\psi_\beta\,g^\beta)\,\geq\,0
\end{eqnarray*}
where the last inequality follows from
\[\M_d(g^\beta\,\z)\succeq0\:\Leftrightarrow\:L_\z(h^2\,g^\beta)\,\geq\,0,\quad\forall h\in\R[\x,t]_d,\]
for every $\beta\in\{0,1\}^m$. 
But since $L_\z(1+t^{2d})\geq0$ and $\epsilon>0$ was arbitrary, one may conclude that
$L_\z(f)\geq0$ for every $f\in\R[\x,t]$ which is nonnegative on $\K\times\R$. Hence by the Riesz-Haviland
theorem (see e.g. \cite[Theorem 3.1, p. 53]{lasserre-book}), $\z$ has a representing measure on $\K\times\R$.
\end{proof}

\subsection{Main result}

Let $L_\infty(\K)$ be the Lebesgue space of integrable functions on $\K$
(with respect to the Lebesgue measure $\lambda$ on $\K$, scaled to a probability measure) and essentially bounded on $\K$.
And with $1\leq p<\infty$,  let $L_p(\K)$ be the Lebesgue space of integrable functions $f$ on $\K$
such that $\int_\k\vert f\vert^p\lambda(d\x)<\infty$. A Borel measure $\mu$ absolutely continuous w.r.t. $\lambda$ is denoted $\mu\ll\lambda$.

\begin{thm}
\label{thmain}
Let $\K\subset \R^n$ be a nonempty compact basic semi-algebraic set of the form
\[\K\,:=\,\{\,\x\in\R^n\::\: g_j(\x)\,\geq\,0,\quad j=1,\ldots,m\,\}\]
for some polynomials $(g_j)\subset\R[\x]$, and recall the notation $g^\beta\in\R[\x]$, with
\[\x\,\mapsto\,g^\beta(\x)\,:=\,g_1(\x)^{\beta_1}\cdots g_m(\x)^{\beta_m},\quad \x\in\R^n,\quad \beta\in\{0,1\}^m.\]
Let $\y=(\y_\alpha)$, $\alpha\in\N^n$, be a real sequence
with $y_0=1$.  
Then the  following two propositions (i) and (ii) are equivalent:

(i) $\y$ has a representing Borel probability measure $\mu\ll\lambda$ on $\K$, 
with a  density  in $\cap_{p\geq1} L_p(\K)$.

(ii) $\M_d(\y)\succeq0$ and $\M_d(g^{\beta}\,\y)\succeq0$ for all $\beta\in\{0,1\}^m$ and all $d\in\N$. In addition, there exists a sequence $\z=(z_{\alpha k})$, $(\alpha,k)\in\N^n\times\N$, such that
(\ref{strip-moment}) holds, and
\begin{equation}
\label{aux}
z_{\alpha 0}\,=\,\int_\K\x^\alpha\,\lambda(d\x)\,;\quad z_{\alpha 1}\,=\,y_\alpha,\quad \forall \alpha\in\N^n.
\end{equation}
Moreover, if in (\ref{aux}) one includes the additional condition $\sup_kz_{0k}<\infty$, then (ii)
is necessary and sufficient for $\y$
to have a representing Borel probability measure $\mu\ll\lambda$ on $\K$, 
with a density in $L_\infty(\K)$.
\end{thm}
\begin{proof}
The $\mbox{(i) $\Rightarrow$ (ii)}$ implication.
As $\y$ has a representing Borel probability measure $\mu$ on $\K$ with 
a density $f\in L_p(\K)$ for every $p=1,2,\ldots$, one may write
\[\mu(A)\,=\,\int_Af(\x)\,\lambda(d\x),\qquad \forall A\,\in\,\mathcal{B}(\R^n).\]
Define the stochastic kernel $\varphi(B\vert\x)$, $B\in\mathcal{B}(\R)$, $\x\in\K$,
where for almost all $\x\in\K$, $\varphi(\cdot\,\vert\,\x)$ is the Dirac measure at the point $f(\x)$.
Next, let $\nu$ be the finite Borel measure on $\K\times\R$ defined by
\begin{equation}
\label{nu}
\nu(A\times B)\,:=\,\int_A\varphi(B\vert\x)\,\lambda(d\x),\quad\forall A\in\mathcal{B}(\R^n),\:B\in\mathcal{B}(\R).\end{equation}
Let $\z=(y_{\alpha k})$, $(\alpha,k)\in\N^n\times\N$, be the sequence of moments of $\nu$.
\begin{eqnarray}
\nonumber
z_{\alpha k}&=&\int_\R \x^\alpha\,t^k\,d\nu(\x,t)\,=\,\int_\K\x^\alpha\,\left(\int_\R t^k\varphi(dt\,\vert\,\x)\right)\,\lambda(d\x),\\
\label{z-alphak}
&=&\int_\K\x^\alpha\,f(\x)^k \,\lambda(d\x)\quad\mbox{(well defined as $f\in L_p(\K)$ for all $p$)}.
\end{eqnarray}
In particular, for every $\alpha\in\N^n$,
\[z_{\alpha 0}\,=\,\int_\K \x^\alpha\, \,\lambda(d\x)\,;\quad z_{\alpha 1}\,=\,\int_\K \x^\alpha\, f(\x)\,\lambda(d\x)\,=\,
\int_\K\x^\alpha\,d\mu\,=\,y_\alpha.\]
Moreover, as $\nu$ is supported on $\K\times\R$ then $\M_d(\y)\succeq0$ and $\M_d(g^\beta\,\y)\succeq0$ for all
$d$ and all $\beta\in\{0,1\}^m$. Hence (\ref{strip-moment})-(\ref{aux}) hold.\\

The (ii) $\Rightarrow$ (i) implication. Let $\z=(z_{\alpha k})$ be such that (\ref{strip-moment})-(\ref{aux}) hold.
By Corollary (\ref{mar-moment}), $\z$ has a representing Borel probability measure $\nu$ on $\K\times\R$.
One may disintegrate $\nu$ in the form 
\[\nu(A\times B)\,=\,\int_{A\cap\K}\varphi(B\,\vert\,\x)\,\psi(d\x),\qquad B\in\mathcal{B}(\R),\:A\in\mathcal{B}(\R^n),\]
for some stochastic kernel $\varphi(\cdot\,\vert\,\x)$, and where $\psi$ is the marginal (probability measure) of $\nu$ on $\K$.
From (\ref{aux}) we deduce that
\[\int_\K\x^\alpha\,\psi(d\x)\,=\,z_{\alpha 0}\,=\,\int_\K\x^\alpha\,\lambda(d\x),\qquad \forall\alpha\in\N^n,\]
which, as $\K$ is compact, implies that $\psi=\lambda$.
In addition, still from (\ref{aux}),
\begin{eqnarray}
\nonumber
z_{\alpha 1}\,=\,\int_\K\x^\alpha\,t\,d\nu(\x,t)&=&\int_\K\x^\alpha\underbrace{\left(\int_\R t\,\varphi(dt\vert\x)\right)}_{f(\x)}\,\lambda(d\x)\,\qquad \forall\alpha\in\N^n\\
\label{signed}
&=&\int_\K\x^\alpha\,\underbrace{f(\x)\,\lambda(d\x)}_{d\theta(\x)}\,\qquad \forall\alpha\in\N^n,
\end{eqnarray}
where $f:\K\to\R$ is the measurable function $\x\mapsto\int_\R t\,\varphi(dt\vert\x)$, and $\theta$ is the signed Borel measure
$\theta(B):=\int_{\K\cap B}f(\x)\,\lambda(d\x)$, for all $B\in\mathcal{B}(\R^n)$. 

But as $\K$ is compact, by Schm\"udgen's Positivstellensatz \cite{schmudgen}, the conditions 
\[\M_d(\y)\,\succeq0,\quad \M_d(g^\beta\,\y)\,\succeq\,0,\quad\beta\in\{0,1\}^m,\quad\forall d\in\N,\]
imply that $\y$ has a finite representing Borel probability measure $\mu$ on $\K$. And so as 
$z_{\alpha 1}=y_\alpha$ for all $\alpha\in\N^n$, and measures on compact sets are moment determinate, 
one may conclude that $d\mu=d\theta=f\,d\lambda$, that is, $\mu\ll\lambda$ on $\K$ (and $f\geq0$ almost everywhere on $\K$).
Next, observe that for every even $p\in\N$, using Jensen's inequality,
\begin{eqnarray*}
z_{0 p}\,=\,\int_\K t^pd\nu(\x,t)&=&\int_\K\underbrace{\left(\int_\R t^p\,\varphi(dt\vert\x)\right)}_{\geq f(\x)^p}\,\lambda(d\x)\,\qquad \forall\alpha\in\N^n\\
&\geq&\int_\K f(\x)^p\,\lambda(d\x)\,\qquad \forall\alpha\in\N^n,
\end{eqnarray*}
and so $f\in L_p(\K)$ for all even $p\geq1$ (hence all $p\in\N$).

Finally consider (\ref{aux}) with the additional condition $\sup_{p}z_{0p}<\infty$.
Then  in the above proof of (i) $\Rightarrow$ (ii) 
 and since now $\y$ has a {\it finite} representing Borel measure with a density 
$f\in L_\infty(\K)$, one has $\lim_{p\to\infty} \Vert f\Vert_p=\Vert f\Vert_\infty$
because $\K$ is compact; see e.g. Ash \cite[problem 9, p. 91]{ash}. And therefore 
since $z_{0p}=\int_\K f(\x)^p\lambda(d\x)$, we obtain $\sup_pz_{0p}<\infty$. 

Similarly,  in the above proof of (ii) $\Rightarrow$ (i), $\sup_{p}z_{0p}<\infty$ implies
$\sup_p\int_\K f(\x)^p\lambda(d\x)=\sup_p\Vert f\Vert_p<\infty$. But this implies that $f\in L_\infty(\K)$
since $\K$ is compact.
\end{proof}

\subsection*{Computational procedure}

Let $\gamma=(\gamma_\alpha)$, $\alpha\in\N^n$, the moment of the Lebesgue measure on $\K$, scaled to make it a	
probability measure. In fact, the (scaled) Lebesgue measure on any box that contains $\K$ is fine.

Let $\y=(y_\alpha)$, $\alpha\in\N^n$, be a real given sequence, and with
$\K$ as in (\ref{setk}) let $v_j:=\lceil({\rm deg}\, g_j)/2\rceil$, $j=1,\ldots,m$.
To check the conditions in Theorem \ref{thmain}(ii), one solves the hierarchy of optimization problems,
parametrized by $d\in\N$.
\begin{equation}
\label{compute}
\begin{array}{rl}
\rho_d=\displaystyle\min_\z&{\rm trace}(\M_d(\z))\\
\mbox{s.t.}&\M_d(\z)\,\succeq\,0\\
&\M_{d-v_j}(g^\beta\,\z)\,\succeq\,0,\quad\beta\in\{0,1\}^m\\
&z_{\alpha 0}\,=\,\gamma_\alpha,\quad \alpha\in\N^n_{2d}\\
&z_{\alpha 1}\,=\,y_\alpha,\quad (\alpha,1)\in\N^{n+1}_{2d}.
\end{array}
\end{equation}
Each problem (\ref{compute}) is a semidefinite program\footnote{A semidefinite program is a convex optimization problem that can be solved efficiently, i.e.,
up to arbitrary fixed precision it can be solved in time polynomial in the input size of the problem; see e.g.
\cite{anjos}.}. Moreover, if (\ref{compute}) has a feasible solution then it has an optimal solution.
This is because as one minimizes the trace of $\M_d(\z)$, the feasible set is bounded and closed, hence compact. 

In (\ref{compute}) one may also include 
the additional constraints $z_{0k}<c$, $k\leq 2d$, for some fixed $c>0$. Then by Theorem \ref{thmain}, $\y$ has a representing Borel probability measure on $\K$ with a density in $L_\infty(\K)$
bounded by $c$, if and only if $\rho_d<\infty$ for all $d$.

Each semidefinite program of the hierarchy (\ref{compute}), $d\in\N$, has a dual which is also a semidefinite program 
and which reads:
\begin{equation}
\label{compute-dual}
\begin{array}{rl}
\rho^*_d=\displaystyle\max_{p,q,\sigma_j}&\displaystyle\int_\K p(\x)\lambda(d\x)+L_\y(q)\\
\mbox{s.t.}&\displaystyle\sum_{(\alpha,k)\in\N^{n+1}_d}(\x^\alpha t^k)^2-(p(\x)+tq(\x))\,=\,\sigma_0(\x,t)+\displaystyle\sum_{j=1}^m\sigma_j(\x,t)g_j(\x)\\
&{\rm deg}\,p\leq 2d;\:{\rm deg}\,q\leq 2d-1;\:\sigma_j\in\Sigma[\x,t]_{t-v_j},\:j=0,\ldots,m,
\end{array}
\end{equation}
where $v_0=0$. In particular, if $\y$ is the sequence of a Borel measure on $\K$ then in (\ref{compute-dual}) one may replace $L_\y(q)$ 
with $\displaystyle\int_\K q(\x)d\mu(\x)$.

\subsection{On membership in $L_p(\K)$}

An interesting feature of
the hierarchy of semidefinite programs (\ref{compute}), $d\in\N$, 
is that it can be used to 
detect if a given sequence $\y=(y_\alpha)$, $\alpha\in\N^n$, cannot have a representing
Borel measure on $\K$ with a density in $L_p(\K)$, $p>1$.
\begin{cor}
Let $\K\subset \R^n$ be as in (\ref{setk}) and let $\y=(\y_\alpha)$, $\alpha\in\N^n$, be a real sequence
with $y_0=1$.  If the semidefinite program (\ref{compute}) with $d\in\N$, has no solution 
then $\y$ cannot have a representing finite Borel measure on $\K$ with a density in
$L_{p}(\K)$, for any $p\geq 2d$.
\end{cor}
\begin{proof}
Suppose that $\y$ has a representing measure on $\K$ with a density $f\in L_{2d}(\K)$,
and hence in $L_k(\K)$ for all $k\leq 2d$.
Proceeding as in the proof of Theorem \ref{thmain}, let $\nu$ be the Borel
measure on $\K\times\R$ defined in (\ref{nu}). Then from (\ref{z-alphak})
one obtains
\[z_{\alpha k}\,=\,\int_\K\x^\alpha\,f(\x)^k\,\lambda(d\x),\quad (\alpha,k)\in\N^{n+1}_{2d},\]
which is well-defined since $\K$ is compact 
(so that $\x^\alpha$ is bounded) and $k\leq 2d$. 
And so the sequence $\z=(z_{\alpha k})$, $(\alpha,k)\in\N^{n+1}_{2d}$,
is a feasible solution of (\ref{compute}) with $d$.
\end{proof}
Notice that again, the detection of absence of a density in $L_p(\K)$
is possible with no {\it \`a priori} bounding parameter $c$. But of course,
the condition is only sufficient.
\begin{ex}
{\rm Let $\K:=[0,1]$ and $s\in [0,1]$. Let $\lambda$ be the Lebesgue measure 
on $[0,1]$ and let $\delta_s$ be the Dirac measure at $s$.
One wants to detect that the Borel probability measure 
$\mu_a:=a\lambda+(1-a)\delta_s$, with $a\in (0,1)$ has no density
in $L_\infty(\K)$. Then (\ref{aux}) reads
\[z_{k0}\,=\,\frac{1}{k+1},\:k=0,1,\ldots;\quad z_{k1}
\,=\,\frac{a}{k+1}+(1-a)s^k,\:k=0,1,\ldots\]
The set $\K$ is defined by $\{x:g(x)\geq0\}$ with $x\mapsto g(x):=x(1-x)$.
We have tested the conditions $\M_d(\z)\succeq0$ and $\M_d(g\,\z)\succeq0$
along with (\ref{aux}) where $k\leq 2d$ (for $z_{0k}$) and
$k\leq 2d-1$ (for $z_{k1}$).

We have considered a Dirac at the points $s=k/10$, $k=1,\ldots,10$,
and with weights $a=1-k/10$, $k=1,\ldots,10$. To solve (\ref{compute})
we have used the GloptiPoly software of Henrion et al. \cite{gloptipoly} dedicated to solving there generalized problem of moments.
Results are displayed in 
Table \ref{tab1} which should be read as follows: 

$\bullet$ A column is parametrized by the 
number of moments involved in the conditions (\ref{aux}). For instance, Column ``10" refers to (\ref{aux})
with $d=10/2$, that is, the moment matrix $\M_d(\z)$ involves moments $z_{ij}$ with $i+j\leq 10$,
i.e., moments up to order $10$.

$\bullet$ Each row is indexed by the location of the Dirac $\delta_s$, $s\in [0,1]$
(with $\mu_a=a\lambda+(1-a)\delta_s$). The statement ``$1-a\geq 0.5$" in row ``$s=0.3$" and column ``10" means that
(\ref{aux}) is violated whenever $1-a\geq 0.5$, i.e., when the weight associated to the Dirac $\delta_s$ is 
larger than $0.5$.

One may see that no matter where the point $s$ is located in the interval $[0,1]$,
if its weight $1-a$ is above $0.5$ then
detection of impossibility of  a density in $L_\infty([0,1])$ occurs 
with moments up to order $10$. If its weight $1-a$ is only above $0.1$ then
detection of impossibility occurs with moments up to order $12$.
So even with a small weight on the Dirac $\delta_s$,  detection of impossibility does not
require moments of order larger than $12$.

\begin{table}
\label{tab1}
\begin{center}
\begin{tabular}{|c|c|c|c|c|}
\hline
$s\,\setminus\mbox{moments}$& 8 & 10 & 12&14\\
\hline
&&&&\\
0.0&$1-a\geq 0.3$ & $1-a\geq 0.1$ & $1-a\geq 0.1$ &$1-a\geq 0.1$\\
0.1&$1-a\geq 1$ & $1-a\geq 0.3$ & $1-a\geq 0.2$ &$1-a\geq 0.2$\\
0.2&$1-a\geq 1$ & $1-a\geq 0.3$ & $1-a\geq 0.1$ &$1-a\geq 0.1$\\
0.3&$1-a\geq 1$ & $1-a\geq 0.5$ & $1-a\geq 0.2$ &$1-a\geq 0.1$\\
0.4&$1-a\geq 1$ & $1-a\geq 0.5$ & $1-a\geq 0.2$ &$1-a\geq 0.1$\\
0.5&$1-a\geq 1$ & $1-a\geq 0.5$ & $1-a\geq 0.2$ &$1-a\geq 0.1$\\
0.6&$1-a\geq 1$ & $1-a\geq 0.5$ & $1-a\geq 0.2$ &$1-a\geq 0.1$\\
0.7&$1-a\geq 1$ & $1-a\geq 0.5$ & $1-a\geq 0.1$ &$1-a\geq 0.1$\\
0.8&$1-a\geq 1$ & $1-a\geq 0.5$ & $1-a\geq 0.1$ &$1-a\geq 0.1$\\
0.9&$1-a\geq 1$ & $1-a\geq 0.5$ & $1-a\neq 0.4,0.5$ &$1-a\geq 0.1$\\
1.0&$1-a\geq 0.4$ & $1-a\geq 0.1$ & $1-a\geq 0.1$ &$1-a\geq 0.1$\\
&&&&\\
\hline
\end{tabular}
\end{center}
\caption{Moments required for detection of failure; one Dirac}
\end{table}
}\end{ex}
\begin{ex}
\label{ex2}
{\rm Still with $\K=[0,1]$, consider now the case where $\mu_a=a\lambda +(1-a)(\delta_{s}+\delta_{s+0.1})/2$, 
that is, $\mu_a$ is a $(a,1-a)$ convex combination of the uniform probability distribution on $[0,1]$
with two Dirac measures at the points $s$ and $s+0.1$ of $[0,1]$, with equal weights. The results displayed in Table \ref{tab2}
are qualitatively very similar to the results in Table \ref{tab1} for
the case of one Dirac.
\begin{table}
\label{tab2}
\begin{center}
\begin{tabular}{|c|c|c|}
\hline
$(s,s+0.1)\,\setminus\mbox{moments}$& 10 & 12\\
\hline
&&\\
(0.1,0.2)&$1-a\geq 0.4$ & $1-a\geq 0.2$\\
(0.2,0.3)&$1-a\geq 0.6$ & $1-a\geq 0.2$\\
(0.3,0.4)&$1-a\geq 0.5$ & $1-a\geq 0.2$\\
(0.4,0.5)&$1-a\geq 0.6$ & $1-a\geq 0.1$\\
(0.5,0.6)&$1-a\geq 0.6$ & $1-a\geq 0.2$\\
(0.6,0.7)&$1-a\geq 0.7$ & $1-a\geq 0.2$\\
(0.7,0.8)&$1-a\geq 0.6$ & $1-a\geq 0.3$\\
(0.8,0.9)&$1-a\geq 0.5$ & $1-a\geq 0.2$\\
(0.9,1.0)&$1-a\geq 0.1$ & $1-a\geq 0.1$\\
&&\\
\hline
\end{tabular}
\end{center}
\caption{Moments required for detection of failure; two Dirac}
\end{table}

}\end{ex}

\end{document}